\documentclass[11pt,a4paper]{article}
\usepackage[margin=25mm,headheight=14pt]{geometry}
\usepackage{fontspec}
\usepackage{amsmath,amssymb,amsthm,mathtools,tikz}
\usepackage[hidelinks,unicode]{hyperref}
\usepackage{microtype}
\allowdisplaybreaks[2]
\newtheorem{theorem}{Theorem}[section]
\newtheorem{proposition}[theorem]{Proposition}
\newtheorem{lemma}[theorem]{Lemma}
\newtheorem{corollary}[theorem]{Corollary}
\theoremstyle{remark}
\newtheorem{remark}[theorem]{Remark}
\theoremstyle{plain}
\newtheorem{maintheorem}{Theorem}

\DeclareMathOperator{\Mod}{Mod}
\DeclareMathOperator{\Sp}{Sp}
\DeclareMathOperator{\SL}{SL}
\DeclareMathOperator{\GL}{GL}
\DeclareMathOperator{\Aut}{Aut}
\DeclareMathOperator{\Out}{Out}
\DeclareMathOperator{\SAut}{SAut}
\DeclareMathOperator{\SOut}{SOut}

\DeclareMathOperator{\tr}{tr}
\newcommand{\Z}{\mathbb Z}
\newcommand{\Q}{\mathbb Q}
\newcommand{\Mg}{\mathcal M_g}
\hypersetup{pdftitle={Infinitely many Nielsen classes in mapping class groups and free-group automorphism groups},pdfauthor={Wenchang Yang}}
\title{Infinitely many Nielsen classes in mapping class groups and free-group automorphism groups}
\author{Wenchang Yang\thanks{Email: \href{mailto:ywc25@mails.tsinghua.edu.cn}{\texttt{ywc25@mails.tsinghua.edu.cn}}}}
\date{}
\begin{document}
\maketitle
\begin{abstract}
We formulate a root-conjugation criterion for producing generating pairs
and detecting Nielsen inequivalence by a polynomial symmetric-trace invariant. The method
simultaneously treats mapping class groups and automorphism groups of free
groups. We prove that $\Mod(\Sigma_g)$, $g\ge2$, and $\Aut(F_n)$,
$\SAut(F_n)$, $\Out(F_n)$ and $\SOut(F_n)$, $n\ge3$, have countably
infinitely many Nielsen classes of generating pairs. A sign parameter
unifies the full and special automorphism-group constructions, with common generation arguments
and trace polynomials. The same homology representations give the
corresponding results for $\Sp_{2g}(\Z)$, $g\ge2$, and $\SL_n(\Z)$ and
$\GL_n(\Z)$, $n\ge3$. This answers Hirose--Monden's low-genus question
and both parts of Kapovich's Question~8.1.
\end{abstract}

\section{Introduction}\label{sec:intro}
An ordered pair $(u,v)$ in a group $G$ is a \emph{generating pair} if
$G=\langle u,v\rangle$. Two pairs are \emph{Nielsen equivalent} if one
can be obtained from the other by finitely many coordinate interchanges,
inversions, and right multiplications of one coordinate by the other.
We construct infinite families of inequivalent generating pairs by a single
mechanism: conjugate a root while retaining a fixed power that recovers
generation, and separate the pairs by a polynomial in a homology trace.
\begin{maintheorem}\label{thm:main}
The following groups have countably infinitely many Nielsen classes of
generating pairs:
\begin{enumerate}
\item $\Mod(\Sigma_g)$ and $\Sp_{2g}(\Z)$ for every $g\ge2$;
\item $\Aut(F_n)$, $\SAut(F_n)$, $\Out(F_n)$, $\SOut(F_n)$,
      $\SL_n(\Z)$ and $\GL_n(\Z)$ for every $n\ge3$.
\end{enumerate}
Here $\Sigma_g$ is a closed connected oriented surface of genus $g$,
$\Mod(\Sigma_g)$ is its orientation-preserving mapping class group,
and $F_n$ is the free group of rank $n$. The special automorphism and
outer automorphism groups are the kernels of the determinant of their
actions on $H_1(F_n;\Z)$.
\end{maintheorem}
Hirose and Monden proved the assertions for $\Mod(\Sigma_g)$ and
$\Sp_{2g}(\Z)$ when $g\ge8$
\cite[Theorem~1.1 and Corollary~1.2]{HM}. Question~4.3 of the 2026 Kyoto
problem collection asks for the remaining genera $2\le g\le7$
\cite[Section~4, p.~7]{Kyoto}. Theorem~\ref{thm:main}(1) answers this
question and extends both assertions to every $g\ge2$.

For $n\ge8$, Kapovich proved the assertions for $\SL_n(\Z)$ and
$\GL_n(\Z)$ \cite[Theorem~A]{Kapovich}, and for $\Out(F_n)$ and
$\SOut(F_n)$ \cite[Theorem~B]{Kapovich}. His generating pairs in
$\Aut(F_n)$ and $\SAut(F_n)$ have inequivalent homology images, so his
construction also gives infinitude for both automorphism groups in this
range \cite[Proposition~4.4 and Theorem~5.5]{Kapovich}.
Theorem~\ref{thm:main}(2) extends all six assertions to every $n\ge3$.
In particular, it answers both parts of Kapovich's Question~8.1 for
every $3\le n\le7$.

Kapovich's construction also uses parameterized conjugation, fixed powers
that recover generators, and commutator traces that distinguish Nielsen
classes. The criterion below isolates a related mechanism adapted to roots
of fixed elements. Our constructions realize this mechanism throughout the
remaining low-rank and low-genus ranges. In rank three, the special base
pair is Nielsen equivalent to Kapovich's pair; the additional ingredient
is the recovery of a conjugator that produces an infinite family
(Section~\ref{subsec:rank3}).

The two principal applications concern mapping class groups and free-group
automorphisms. For surfaces, Korkmaz's
generating pair \cite[Theorem~5]{Korkmaz} and the twist roots of Margalit
and Schleimer \cite[p.~1496]{MS} provide the generation data; homology
produces a quartic symmetric trace with leading coefficient $2$.
For free groups, a sign parameter gives one construction for the full
automorphism group and its special subgroup in ranks at least four.
A second signed construction handles rank three. In each range, both
signs use the same conjugator and the same symmetric trace polynomial.
The rank-three conjugator is recovered by an exact word identity in the
automorphism group: a product cancels the inner error in an intermediate
commutator. Projection then gives the full and special outer assertions
without further generation or trace calculations.
Hirose and Monden also proved that $\Mg$ has infinitely many
T-equivalence classes of generating pairs for $g\ge8$
\cite[Theorem~1.4]{HM}.
Corollary~\ref{cor:T-systems} extends this result to every $g\ge2$
(Corollary~\ref{cor:T-systems}).

The arithmetic groups occur as homological quotients of these two
families. The representations used to separate the original pairs also
separate their quotient images. Thus no independent matrix-generation
argument is needed for $\Sp_{2g}(\Z)$, $\SL_n(\Z)$ or $\GL_n(\Z)$.

For a finite-dimensional representation $\rho:G\to\GL(V)$, define
\begin{equation}\label{eq:Jrho}
J_\rho(u,v)=\tr\rho([u,v])+\tr\rho([u,v]^{-1}),\qquad
[u,v]=uvu^{-1}v^{-1}.
\end{equation}
The commutator of a Nielsen-equivalent pair is conjugate to the original
commutator or its inverse. Hence $J_\rho$ is a Nielsen invariant. The
following criterion combines generation with polynomial separation.

\begin{maintheorem}[Root-conjugation criterion]\label{thm:criterion}
Let $G=\langle x,r\rangle$, let $m\ge2$, and suppose that
\begin{equation}\label{eq:group-hypotheses}
q\in\langle x,r^m\rangle,\qquad [q,r^m]=1.
\end{equation}
Put $r_t=q^trq^{-t}$ for $t\in\Z$.
\begin{enumerate}
\item Every pair $(x,r_t)$ generates $G$.
\item Let $\rho:G\to\GL(V)$ be a finite-dimensional representation over
a field $K$ of characteristic zero, and suppose that $\rho(q)=I+N$, $N^2=0$.
Then $P(t)=J_\rho(x,r_t)$ belongs to $K[t]$ and has degree at most four.
For $A=\rho(x)$ and $B=\rho(r)$, define
\begin{equation}\label{eq:trace-bilinear}
\Phi_A(U,W)=\tr(AUA^{-1}W)+\tr(A^{-1}UAW).
\end{equation}
Its coefficient of $t^4$ is
\[
[t^4]P(t)=\Phi_A(NBN,NB^{-1}N).
\]
If $P$ is nonconstant, the pairs $(x,r_t)$ represent infinitely many
Nielsen classes. In particular, a nonzero quartic coefficient suffices.
\end{enumerate}
\end{maintheorem}

Section~\ref{sec:criterion} proves the criterion through two lemmas.
Sections~\ref{sec:surface} and~\ref{sec:free} implement it for surfaces
and free groups, respectively, and deduce the arithmetic consequences
from the same homology calculations.

\noindent\textbf{Conventions.}
Products act on the rightmost factor first, vectors are columns, and
all Dehn twists are right twists. We write $E_{ij}$ for a matrix unit,
$I_k$ for the identity matrix, and $[t^d]P$ for a polynomial coefficient;
empty identity blocks are omitted. Curves and homeomorphisms also denote
their isotopy classes. The intersection sign is fixed in
\eqref{eq:twist-action}.

\section{The root-conjugation mechanism}\label{sec:criterion}
Generation is recovered from a fixed power; separation is checked in a
representation. We prove these two statements separately.

\begin{lemma}[Generation]\label{lem:generation}
Let $G=\langle x,r\rangle$ and let $m\ge2$. If
$q\in\langle x,r^m\rangle$ and $[q,r^m]=1$, then
\[
\langle x,q^trq^{-t}\rangle=G\qquad\text{for every }t\in\Z.
\]
\end{lemma}
\begin{proof}
Put $r_t=q^trq^{-t}$ and $H_t=\langle x,r_t\rangle$.
Since $q$ commutes with $r^m$,
\[
r_t^m=q^tr^mq^{-t}=r^m.
\]
Hence $\langle x,r^m\rangle\subseteq H_t$, so $q\in H_t$.
It follows that $r=q^{-t}r_tq^t\in H_t$. Thus $H_t$ contains
both generators of $G$, and $H_t=G$.
\end{proof}

\begin{lemma}[Separation]\label{lem:trace}
Let $\rho:G\to\GL(V)$ be a finite-dimensional representation over a
field $K$ of characteristic zero.
\begin{enumerate}
\item The function $J_\rho$ is invariant under Nielsen moves.
Any family of ordered pairs on which $J_\rho$ takes infinitely many
values contains infinitely many Nielsen classes.

\item Suppose that $\rho(q)=I+N$ with $N^2=0$.
For $x,r\in G$, put $A=\rho(x)$, $B=\rho(r)$, and use
$\Phi_A$ from \eqref{eq:trace-bilinear}. Set
\begin{equation}\label{eq:squarezero-expansion}
\begin{aligned}
U(t)&=B+t(NB-BN)-t^2NBN,\\
W(t)&=B^{-1}+t(NB^{-1}-B^{-1}N)-t^2NB^{-1}N.
\end{aligned}
\end{equation}
Then
\begin{equation}\label{eq:trace-polynomial}
J_\rho(x,q^trq^{-t})=P(t):=\Phi_A(U(t),W(t)),\qquad \deg P\le4,
\end{equation}
with
\begin{equation}\label{eq:quartic-test}
[t^4]P(t)=\Phi_A(NBN,NB^{-1}N).
\end{equation}
\end{enumerate}
\end{lemma}
\begin{proof}
For $c=[u,v]$, direct multiplication gives
\[
\begin{gathered}
[v,u]=c^{-1},\qquad [u^{-1},v]=u^{-1}c^{-1}u,\qquad
[u,v^{-1}]=v^{-1}c^{-1}v,\\
[uv,v]=c,\qquad [u,vu]=c.
\end{gathered}
\]
These identities cover the elementary Nielsen moves.
Trace is unchanged by conjugation, and the symmetric sum in
\eqref{eq:Jrho} is unchanged by inversion. Thus $J_\rho$ is constant
on each Nielsen class, proving part~(1).

Since $N^2=0$, for every integer $t$ we have
$(I+N)^t=I+tN$ and $(I+N)^{-t}=I-tN$.
Writing $B_t=\rho(q^trq^{-t})$, multiplication gives
\[
B_t=(I+tN)B(I-tN)=U(t),\qquad
B_t^{-1}=(I+tN)B^{-1}(I-tN)=W(t).
\]
Cyclic invariance of trace now gives
\[
\begin{aligned}
J_\rho(x,q^trq^{-t})
&=\tr(AB_tA^{-1}B_t^{-1})
  +\tr(B_tAB_t^{-1}A^{-1})\\
&=\Phi_A(B_t,B_t^{-1})=\Phi_A(U(t),W(t)).
\end{aligned}
\]
Both $U$ and $W$ have degree at most two. Bilinearity gives
$\deg P\le4$, and their quadratic terms contribute
$\Phi_A(-NBN,-NB^{-1}N)=\Phi_A(NBN,NB^{-1}N)$ to $t^4$.
This proves part~(2).
\end{proof}

\begin{remark}[Finite differences]\label{rem:differences}
More generally, for $Q=\rho(q)$ and $\Delta(Z)=QZQ^{-1}-Z$, the
conditions $\Delta^{a+1}(B)=\Delta^{b+1}(B^{-1})=0$ give a trace
polynomial of degree at most $a+b$, by substituting the finite binomial
expansions of $(I+\Delta)^tB$ and $(I+\Delta)^tB^{-1}$ into
$\Phi_A$. We use only the square-zero case above.
\end{remark}

\begin{proof}[Proof of Theorem~\ref{thm:criterion}]
Lemma~\ref{lem:generation} proves part~(1).
For part~(2), a nonconstant polynomial $P$ takes infinitely many
values on the nonnegative integers: characteristic zero makes these
integers distinct in $K$, and each value has at most $\deg P$ preimages.
Lemma~\ref{lem:trace}(1) therefore separates infinitely many of the
pairs, all of which generate $G$ by part~(1).
The polynomial formula and its quartic coefficient follow from
Lemma~\ref{lem:trace}(2).
\end{proof}

\begin{remark}[Homological quotients]\label{rem:quotients}
If $\pi:G\twoheadrightarrow H$ and $\rho=\sigma\circ\pi$, then
generating pairs map to generating pairs and
$J_\rho(u,v)=J_\sigma(\pi(u),\pi(v))$. A nonconstant trace polynomial
therefore proves infinitude in both groups. This factorization is used
for every arithmetic consequence below.
\end{remark}

\section{Mapping class groups and their symplectic quotients}\label{sec:surface}
Here the fixed power is a nonseparating Dehn twist.
Fix $g\ge2$ and write $\Mg=\Mod(\Sigma_g)$.

Realize $\Sigma_g$ as the double of a planar surface $\mathcal P_g$
with outer boundary $B_0$ and $g$ holes $B_1,\ldots,B_g$ arranged in a
row. Let $a_1$ join $B_0$ to the left side of $B_1$, let $a_i$ for
$i\ge2$ join $B_0$ to the top of $B_i$, and let $h_i$ join the adjacent
holes $B_{i+1}$ and $B_i$, as in Figure~\ref{fig:planar}.
Let $\alpha_i,\gamma_i$ be the doubles of $a_i,h_i$, and let $\beta_i$
be the seam $B_i$.
\begin{figure}[htbp]
\centering
\begin{tikzpicture}[x=.72cm,y=.72cm]
\draw[rounded corners=5pt] (0,-2) rectangle (13,2);
\foreach \x/\i in {2/1,5/2,8/3,11/4}{
\draw (\x,0) circle (.5); \node at (\x,-.9) {$B_{\i}$};}
\draw[thick] (0,0)--(1.5,0); \node at (.7,.35) {$a_1$};
\foreach \x/\i in {5/2,8/3,11/4}{
\draw[thick] (\x,.5)--(\x,2); \node[right] at (\x,1.3) {$a_{\i}$};}
\draw[thick] (2.5,0)--(4.5,0) node[midway,above] {$h_1$};
\draw[thick] (5.5,0)--(7.5,0) node[midway,above] {$h_2$};
\draw[thick] (8.5,0)--(10.5,0) node[midway,above] {$h_3$};
\node at (12,-1.55) {$\mathcal P_4$};
\end{tikzpicture}
\caption{One copy of the planar surface; every boundary is glued to the other copy.}
\label{fig:planar}
\end{figure}
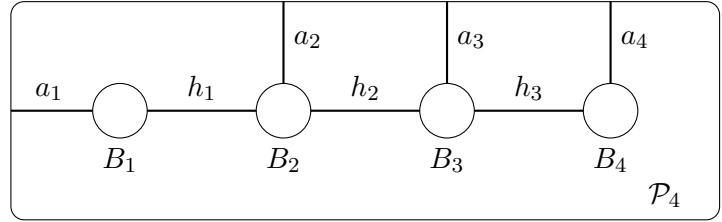
The chain and its twists are
\begin{equation}\label{eq:chain}
c_1=\alpha_1,\quad c_{2i}=\beta_i\ (1\le i\le g),\quad
c_{2i+1}=\gamma_i\ (1\le i<g),\qquad A_j=T_{c_j},\quad D_i=T_{\alpha_i}.
\end{equation}
Adjacent chain curves meet once; all other pairs are disjoint.
In particular, $\alpha_g$ is disjoint from the odd chain
$c_1,\ldots,c_{2g-1}$. Fix the intersection sign by
\begin{equation}\label{eq:twist-action}
(T_c)_*(w)=w+\omega([c],w)[c].
\end{equation}
Orient the curves to obtain the integral symplectic basis
\begin{equation}\label{eq:surface-basis}
e_i=[\alpha_i],\quad f_i=[\beta_i],\quad
\omega(e_i,f_j)=\delta_{ij},\quad
\omega(e_i,e_j)=\omega(f_i,f_j)=0,\quad [\gamma_i]=e_i-e_{i+1}.
\end{equation}
To justify the identity $[\gamma_i]=e_i-e_{i+1}$ in
\eqref{eq:surface-basis}, orient
$a_i$ from $B_0$ to $B_i$ and $h_i$ from
$B_{i+1}$ to $B_i$. The boundary map from relative $H_1(\mathcal P_g,
\partial\mathcal P_g;\Z)$ to $H_0(\partial\mathcal P_g;\Z)$ is injective,
since boundary circles generate $H_1(\mathcal P_g;\Z)$. Thus
$[h_i]=[a_i-a_{i+1}]$ in relative homology; doubling gives the identity.
The intersection matrix in \eqref{eq:surface-basis} is unimodular.

\begin{lemma}\label{lem:root}
For the right twists in \eqref{eq:chain},
\begin{equation}\label{eq:surface-root}
k_g=A_1^2A_2\cdots A_{2g-1},\quad k_g^{2g-1}=D_g^2,\qquad
r_g=k_g^gD_g^{-1}=k_g^{1-g}D_g,\quad r_g^{2g-1}=D_g.
\end{equation}
\end{lemma}
\begin{proof}
A regular neighborhood $U$ of the odd chain has genus $g-1$ and two
boundary components. Its $2g-1$ chain classes are independent in
$H_1(\Sigma_g;\Z)$ by \eqref{eq:surface-basis}, so $H_1(U;\Z)$ injects
into $H_1(\Sigma_g;\Z)$. Its complement has Euler characteristic zero
and two boundary components. A disconnected complement would contain
a disk, killing a nonzero boundary class of $U$. Hence the complement
is an annulus. Since $\alpha_g$ lies in this annular complement and
represents a nonzero homology class, it is essential in the annulus.
Hence it is isotopic to either boundary component of $U$.

Writing $\widetilde T_c$ for a left twist, the odd-chain relation of
\cite[p.~1496]{MS}, with $d_1,d_2$ the boundaries of $U$, is
\begin{equation}\label{eq:left-chain-relation}
(\widetilde T_{c_1}^{\,2}\widetilde T_{c_2}\cdots
\widetilde T_{c_{2g-1}})^{2g-1}
=\widetilde T_{d_1}\widetilde T_{d_2}.
\end{equation}
The involution interchanging the two planar copies reverses orientation
and preserves every chain curve and $\alpha_g$. Conjugating
\eqref{eq:left-chain-relation} by this involution changes left twists to right twists in the
same factor order, giving $k_g^{2g-1}=D_g^2$.
Disjointness gives $[k_g,D_g]=1$, so the two root words agree and
$(k_g^gD_g^{-1})^{2g-1}=D_g^{2g}D_g^{-(2g-1)}=D_g$.
\end{proof}

\begin{proposition}\label{prop:surface-data}
There are $x_g,r_g,q_g\in\Mg$ such that
\begin{equation}\label{eq:surface-data}
\langle x_g,D_g\rangle=\langle x_g,r_g\rangle=\Mg,\qquad
r_g^{2g-1}=D_g,\qquad q_g=D_1,\qquad [q_g,D_g]=1.
\end{equation}
One may take $r_g$ from Lemma~\ref{lem:root} and
$x_g=\tau_gs_g\tau_g^{-1}$, where $s_g=A_{2g}\cdots A_1$ and
$\tau_g$ exchanges the oriented marked handles
$(\alpha_2,\beta_2)$ and $(\alpha_g,\beta_g)$, fixing the others;
$\tau_2=1$.
\end{proposition}
\begin{proof}
The correspondence with \cite[Figure~1]{Korkmaz} is
\[
\begin{array}{c|ccc}
\text{Korkmaz} & a_j & b & S\\ \hline
\text{here} & c_j & \alpha_2 & A_{2g}\cdots A_1
\end{array}
\]
In the doubled planar model, the seam curves alternate with the doubles
of $a_1,h_1,\ldots,h_{g-1}$ to form the ordered chain in that figure.
The extra curve $b$ is the double of the arc $a_2$, namely $\alpha_2$.
Korkmaz defines $S$ as the product of right twists in decreasing chain
order, so it corresponds to $s_g=A_{2g}\cdots A_1$, with the same
composition convention. Theorem~5 of that paper gives
$\langle S,T_b\rangle$ as the mapping class group of the once-bordered
surface. Filling its boundary induces a surjective homomorphism onto
$\Mg$ and sends these generators to $s_g,D_2$; hence
$\langle s_g,D_2\rangle=\Mg$.
The handle interchange extends over the complementary sphere with
$g$ boundary components. Conjugating gives
$\langle x_g,D_g\rangle=\Mg$.
Lemma~\ref{lem:root} proves the root assertion and hence
$\langle x_g,r_g\rangle=\Mg$. Finally $\alpha_1$ and $\alpha_g$ are
disjoint, so $[D_1,D_g]=1$.
\end{proof}

We now use the integral homology action
$\rho_g:\Mg\to\Sp_{2g}(\Z)$, also viewed on $H_1(\Sigma_g;\Q)$.
\begin{lemma}\label{lem:pairings}
For $A=\rho_g(x_g)$ and $B=\rho_g(r_g)$,
\begin{equation}\label{eq:pairings}
\omega(e_1,Ae_1)=-1,\qquad \omega(e_1,Be_1)=(-1)^g.
\end{equation}
\end{lemma}
\begin{proof}
Put $S=\rho_g(s_g)$, $L=\rho_g(k_g)$, and $F_j=\sum_{i=1}^j f_i$.
The homology map $\Pi_g$ of $\tau_g$ exchanges the second and $g$-th
basis pairs and fixes $e_1$.
Applying the factors of $s_g$ in increasing chain order takes $e_1$
to $e_j-F_j$ after $T_{\beta_j}$, and finally to $e_g-F_g$. Thus
\begin{equation}\label{eq:A-vector}
Ae_1=\Pi_gSe_1=e_2-F_g.
\end{equation}
For $1\le j<g$, applying the factors of $k_g$ in decreasing chain
order gives
\begin{equation}\label{eq:L-action}
Le_j=-e_1-F_j,\qquad Lf_j=e_j-e_{j+1},\qquad LF_j=e_1-e_{j+1}.
\end{equation}
Indeed, before $T_{\alpha_1}^2$, the vector $e_j$ becomes $e_1-F_j$;
the final square twist gives $-e_1-F_j$.
For $f_j$, the two effective factors give
$f_j\xmapsto{T_{\gamma_j}}f_j+e_j-e_{j+1}
\xmapsto{T_{\beta_j}}e_j-e_{j+1}$, which the remaining
factors fix. Summing the formulas for $Lf_j$ gives
$LF_j=e_1-e_{j+1}$ in \eqref{eq:L-action}.
Alternating induction from $Le_1=-e_1-F_1$ now yields
\begin{equation}\label{eq:surface-orbit}
L^{2j-1}e_1=-e_j-F_j,\qquad L^{2j}e_1=e_{j+1}+F_j
\quad(1\le j\le g-1).
\end{equation}
Since $(D_g)_*$ fixes $e_1$, the expression $r_g=k_g^gD_g^{-1}$ in
\eqref{eq:surface-root} gives $Be_1=L^ge_1$.
For $g=2\ell$ this is $e_{\ell+1}+F_\ell$, and for $g=2\ell+1$ it is
$-e_{\ell+1}-F_{\ell+1}$. Together with \eqref{eq:A-vector}, these expressions
prove \eqref{eq:pairings}, including $g=2$.
\end{proof}

\begin{proposition}\label{prop:surface}
For every $g\ge2$, the pairs $(x_g,q_g^tr_gq_g^{-t})$, $t\in\Z$,
generate $\Mg$ and represent infinitely many Nielsen classes.
Their invariant $J_{\rho_g}$ is a quartic polynomial with leading
coefficient $2$.
\end{proposition}
\begin{proof}
Proposition~\ref{prop:surface-data} verifies
\eqref{eq:group-hypotheses} with $m=2g-1$.
By \eqref{eq:twist-action}, $\rho_g(q_g)=I+N$ with
$N(w)=\omega(e_1,w)e_1$.
Since $N^2=0$, Theorem~\ref{thm:criterion}(2) gives a polynomial
$P(t)=J_{\rho_g}(x_g,q_g^tr_gq_g^{-t})$ of degree at most four.
Put $A=\rho_g(x_g)$, $B=\rho_g(r_g)$, and
$\lambda=\omega(e_1,Ae_1)$, $\mu=\omega(e_1,Be_1)$.
The rank-one expression for $N$, together with preservation of
$\omega$ by $B$, gives
\[
NBN=\mu N,\qquad NB^{-1}N=-\mu N.
\]
Likewise, $ANA^{-1}(w)=\omega(Ae_1,w)Ae_1$ and
$A^{-1}NA(w)=\omega(A^{-1}e_1,w)A^{-1}e_1$.
Using $\tr(w\mapsto\varphi(w)u)=\varphi(u)$, we obtain
\[
\tr(ANA^{-1}N)=\tr(A^{-1}NAN)=-\lambda^2,
\qquad \Phi_A(N,N)=-2\lambda^2.
\]
Equation~\eqref{eq:quartic-test} and Lemma~\ref{lem:pairings} give
\begin{equation}\label{eq:quartic-coefficient}
[t^4]P(t)=\Phi_A(NBN,NB^{-1}N)
=\Phi_A(\mu N,-\mu N)=2\lambda^2\mu^2=2.
\end{equation}
The polynomial therefore has degree exactly four.
Theorem~\ref{thm:criterion} now gives generation and infinitely many
Nielsen classes.
\end{proof}

\begin{corollary}[Symplectic consequence]\label{cor:symplectic}
For every $g\ge2$, the group $\Sp_{2g}(\Z)$ has countably infinitely many
Nielsen classes of generating pairs, obtained as homology images of the
pairs in Proposition~\ref{prop:surface}.
\end{corollary}
\begin{proof}
The integral homology action $\rho_g$ is surjective
\cite[Theorem~6.4]{FM}. Its image pairs therefore generate
$\Sp_{2g}(\Z)$, and the symmetric trace in the defining representation
is the same nonconstant polynomial as in Proposition~\ref{prop:surface}.
The group is countable.
\end{proof}

Two generating pairs of a group $G$ are \emph{T-equivalent} if they are
related by Nielsen moves and simultaneous application of automorphisms
of $G$. The resulting classes are also called \emph{T-systems}.
\begin{corollary}[T-equivalence]\label{cor:T-systems}
For every $g\ge2$, the group $\Mg$ has countably infinitely many
T-equivalence classes of generating pairs.
\end{corollary}
\begin{proof}
Inner automorphisms act trivially on Nielsen classes of generating pairs:
if $f:F(X,Y)\twoheadrightarrow G$ represents a pair and $f(w)=v$,
then $f\circ\kappa_w=\kappa_v\circ f$, where $\kappa$ denotes inner
conjugation. Thus the action of $\Aut(G)$ on Nielsen classes factors
through $\Out(G)$. McCarthy proved that $\Out(\Mg)$ is finite for every
$g\ge2$, of order four for $g=2$ and two for $g\ge3$ \cite{McCarthy}.
Each T-system is consequently a finite union of Nielsen classes.
Proposition~\ref{prop:surface} gives infinitely many such classes, hence
infinitely many T-systems. Countability follows from that of $\Mg$.
\end{proof}

\section{Free-group automorphisms and their homological quotients}\label{sec:free}
Let $F_n=F(a_1,\ldots,a_n)$ and let
$\rho:\Aut(F_n)\twoheadrightarrow\GL_n(\Z)$ be abelianization in the
basis $[a_1],\ldots,[a_n]$. Inner automorphisms act trivially, so $\rho$
factors through $\Out(F_n)$. Put
\[
\SAut(F_n)=\ker(\det\circ\rho),\qquad
\SOut(F_n)=\SAut(F_n)/\operatorname{Inn}(F_n).
\]
Write $R_{ij}:a_i\mapsto a_ia_j$ and $L_{ij}:a_i\mapsto a_ja_i$,
fixing all other basis elements. Gersten's generating theorem
\cite{Gersten} says that all $R_{ij}$ and $L_{ij}$ generate
$\SAut(F_n)$ for $n\ge3$; see also
\cite[Proposition~2.5(3)]{Kapovich}. Their elementary homology matrices
also show that the special groups map onto $\SL_n(\Z)$.

For both constructions below, use the sign parameter
\begin{equation}\label{eq:signed-groups}
G_{n,\delta}=\begin{cases}\SAut(F_n),&\delta=0,\\\Aut(F_n),&\delta=1,
\end{cases}\qquad
\overline G_{n,\delta}=\begin{cases}\SOut(F_n),&\delta=0,\\\Out(F_n),&\delta=1.
\end{cases}
\end{equation}
A bar denotes an outer class. In each rank range, the conjugator and
the trace polynomial are independent of $\delta$.

\subsection{Ranks at least four}\label{subsec:large-rank}
Fix $n\ge4$ and $\delta\in\{0,1\}$, and put
$\varepsilon=(-1)^{(1-\delta)(n-1)}$ and $\eta=(-1)^{1-\delta}$.
Define
\begin{equation}\label{eq:large-generators}
\begin{gathered}
x_\delta(a_i)=a_{i+1}\ (i<n),\qquad x_\delta(a_n)=a_1^\varepsilon,\\
r_\delta(a_1)=a_1a_3,\quad
r_\delta(a_2)=a_3^{-1}a_2a_3,\quad
r_\delta(a_3)=a_3^{-1}a_2,\\
r_\delta(a_4)=a_4^\eta,\qquad r_\delta(a_i)=a_i\ (i\ge5),\qquad q=R_{32}.
\end{gathered}
\end{equation}
On the first three basis elements this is the root of Margalit and
Schleimer \cite[p.~1497]{MS}. Direct substitution gives
\begin{equation}\label{eq:large-root}
r_\delta^2=D:=R_{12},\qquad
\det\rho(x_\delta)=(-1)^{\delta(n-1)},\qquad
\det\rho(r_\delta)=(-1)^\delta.
\end{equation}
Thus $r_\delta$ is an automorphism with inverse $r_\delta D^{-1}$.
The family is
\begin{equation}\label{eq:large-family}
r_{\delta,t}=q^tr_\delta q^{-t},\qquad
\mathcal P_{n,\delta,t}=(x_\delta,r_{\delta,t}).
\end{equation}

\begin{lemma}[Generation]\label{lem:large-generation}
For every $t\in\Z$, the pair $\mathcal P_{n,\delta,t}$ generates
$G_{n,\delta}$.
\end{lemma}
\begin{proof}
First put $K=\langle x_\delta,r_\delta\rangle$ and
$H=\langle x_\delta,D\rangle\le K$. Cyclic conjugation of $D$ gives
$R_{12},R_{23},\ldots,R_{n-1,n}$ and $R_{n1}$, the last possibly
inverted. For distinct $i,j,k$, the exact identity
\begin{equation}\label{eq:right-commutator}
[R_{jk}^{-1},R_{ij}^{-1}]=R_{ik}
\end{equation}
follows by applying the factors from right to left to $a_i$: they give
$a_ia_j$, $a_ia_ja_k$, $a_ia_k$, $a_ia_k$. The full word fixes all
other basis elements. Induction along oriented paths in the cycle
therefore places every $R_{ij}$ in $H$.

Put $u=x_\delta^{-1}r_\delta x_\delta$. The word
\begin{equation}\label{eq:large-swap}
\sigma=D^{-1}u^{-3}DuD^{-2}u
\end{equation}
exchanges $a_1,a_2$ with both inverted, sends $a_3$ to $a_3^\eta$,
and fixes all other $a_i$. Here is a free-reduction certificate valid
for both signs. On $a=a_1$, $b=a_2$, $c=a_n$, we have
\[
u(a)=b^{-1}ab,\quad u(b)=b^{-1}a,\quad
u(c)=\begin{cases}cb,&\varepsilon=1,\\b^{-1}c,&\varepsilon=-1.
\end{cases}
\]
The successive images are
\[
\begin{array}{c|cc|cc}
 &a&b&c\ (\varepsilon=1)&c\ (\varepsilon=-1)\\ \hline
DuD^{-2}u&a^{-1}b&a^{-2}b&cab&b^{-1}a^{-1}c\\
u^{-3}DuD^{-2}u&b^{-1}&b^{-1}a^{-1}&c&c\\
\sigma&b^{-1}&a^{-1}&c&c.
\end{array}
\]
The map $u$ also sends $a_3$ to $a_3^\eta$ and fixes the remaining
basis elements. Since $D$ fixes $a_3$ and the exponent sum of $u$ in
the word \eqref{eq:large-swap} is $-1$, we obtain
$\sigma(a_3)=u^{-1}(a_3)=a_3^\eta$.

Since $\sigma^2=1$, substitution gives
\begin{equation}\label{eq:left-recovery}
\sigma R_{21}\sigma^{-1}=L_{12}.
\end{equation}
The underlying permutations of $x_\delta,\sigma$ are an $n$-cycle and
the transposition $(1\,2)$, so their image is $S_n$. Given $i\ne j$,
choose a signed permutation in $\langle x_\delta,\sigma\rangle$ sending
the ordered pair $(1,2)$ to $(i,j)$. Conjugating $R_{12}$ and $L_{12}$
by it gives both $R_{ij}$ and $L_{ij}$, up to inversion and interchange
of left and right. Hence Gersten's generating set is contained in $K$.
For $\delta=0$ both generators have determinant $1$, so
$K=\SAut(F_n)$. For $\delta=1$, $K$ also contains the determinant
$-1$ element $r_\delta$, so $K=\Aut(F_n)$.

Finally $q=R_{32}\in H$ and $[q,D]=1$. Theorem~\ref{thm:criterion}(1),
with $m=2$, now proves the statement for every $t$.
\end{proof}

\begin{lemma}[Separation]\label{lem:large-trace}
For both values of $\delta$ and every $n\ge4$,
\begin{equation}\label{eq:large-polynomial}
J_\rho(\mathcal P_{n,\delta,t})=2n-9+4t-4t^2.
\end{equation}
\end{lemma}
\begin{proof}
Let $A_\delta=\rho(x_\delta)$ and $B_{\delta,t}=\rho(r_{\delta,t})$.
Since $\rho(q)=I_n+E_{23}$, we obtain
\begin{equation}\label{eq:large-block}
B_{\delta,t}=
\begin{pmatrix}1&0&0\\t&1&1-2t\\1&0&-1\end{pmatrix}
\oplus(\eta)\oplus I_{n-4},\qquad
B_{\delta,t}^{-1}=B_{\delta,t}(I_n-E_{21}).
\end{equation}
Put $C_{\delta,t}=A_\delta B_{\delta,t}A_\delta^{-1}B_{\delta,t}^{-1}$,
$v=-1+t-2t^2$ and $w=3t-2t^2$. Multiplication gives the diagonals
\[
\begin{array}{c|cc}
 &C_{\delta,t}&C_{\delta,t}^{-1}\\ \hline
n=4&(\eta,1,v,-\eta)&(\eta,w,-1,-\eta)\\
n\ge5&(1,1,v,-\eta,\eta,1,\ldots,1)&(1,w,-1,-\eta,\eta,1,\ldots,1).
\end{array}
\]
In rank four, $\varepsilon=\eta$; in higher ranks the wraparound sign
does not enter these diagonal entries. Thus in every rank
\begin{equation}\label{eq:large-traces}
\tr C_{\delta,t}=n-4+t-2t^2,\qquad
\tr C_{\delta,t}^{-1}=n-5+3t-2t^2.
\end{equation}
Their sum proves \eqref{eq:large-polynomial}.
\end{proof}

\subsection{Rank three}\label{subsec:rank3}
Write $F_3=F(a,b,c)$, fix $\delta\in\{0,1\}$, and put $s=2\delta-1$.
Define
\begin{equation}\label{eq:rank3-generators}
\begin{array}{c|ccc}
 &a&b&c\\ \hline
\alpha&b&a^{-1}&c\\
h_\delta&c^sa&a&b^{-1}\\
q&b^4ab^4&b&c.
\end{array}
\end{equation}
Their inverses are $\alpha^{-1}:(a,b,c)\mapsto(b^{-1},a,c)$,
$h_\delta^{-1}:(a,b,c)\mapsto(b,c^{-1},(ba^{-1})^{-s})$, and
$q^{-1}:a\mapsto b^{-4}ab^{-4}$, fixing $b,c$.
Moreover $\det\rho(\alpha)=1$ and $\det\rho(h_\delta)=(-1)^\delta$.
Put
\begin{equation}\label{eq:rank3-family}
z=\alpha^2,\qquad \alpha_t=q^t\alpha q^{-t},\qquad
\mathcal P_{3,\delta,t}=(h_\delta,\alpha_t).
\end{equation}
The involution $z$ inverts $a,b$ and fixes $c$.

For $\delta=0$, this base pair is Nielsen equivalent to Kapovich's
rank-three pair \cite[equation~(3.9)]{Kapovich}. Indeed his generators
are $(\alpha h_0,h_0)$: applying $\alpha h_0$ to $(a,b,c)$ gives
$(c^{-1}b,b,a)$. Starting from $(h_0,\alpha)$, interchange the
coordinates and multiply the first by the second. The additional input
below is the recovery of $q$ from $h_\delta,\alpha^2$, which extends
the base pair to an infinite family.

\begin{lemma}[Generation]\label{lem:rank3-generation}
For every $t\in\Z$, the pair $\mathcal P_{3,\delta,t}$ generates
$G_{3,\delta}$.
\end{lemma}
\begin{proof}
Let $D:a\mapsto ab$, fixing $b,c$, and put $u=\alpha^{1-2\delta}$.
The unified word identity
\begin{equation}\label{eq:rank3-D-word}
D=u h_\delta^2u h_\delta^{-1}u^{-1}h_\delta^{-1}u
h_\delta^{-1}z^\delta h_\delta^{-3}u^{-1}
\end{equation}
is checked by free reduction. To verify it, write
$V=h_\delta^{-1}u^{-1}h_\delta^{-1}u h_\delta^{-1}z^\delta h_\delta^{-3}u^{-1}$.
The first two rows below record the basis images after applying $V$
for each sign. The last row records their images after applying the
remaining prefix $u h_\delta^2u$, for either sign:
\[
\begin{array}{c|ccc}
 &a&b&c\\ \hline
V\ (\delta=0)&abc^{-1}&c^{-1}&c^{-1}a^{-1}\\
V\ (\delta=1)&a^{-1}b^{-1}c&c&a^{-1}c\\
u h_\delta^2uV&ab&b&c.
\end{array}
\]
Thus $D\in\langle\alpha,h_\delta\rangle$ and
$\tau_\delta=h_\delta D^{-1}$ sends $(a,b,c)$ to $(c^s,a,b^{-1})$.
We have $\tau_0^3=1$, while $\tau_1^3$ inverts all three basis elements.
Put $\beta_\delta=\tau_\delta^{1+\delta}$. Both $\alpha$ and
$\beta_\delta$ are determinant-one signed permutations; their underlying
permutations generate $S_3$. The involution $\alpha^2$ and its
$\beta_\delta$-conjugates give all double inversions, so
$\langle\alpha,\beta_\delta\rangle$ is the full determinant-one signed
permutation subgroup $W_3^+$.

Conjugates of $D$ and $D^{-1}$ by $W_3^+$ give all right and left
Nielsen maps. Given $i\ne j$, choose a signed permutation sending
$a$ to $a_i^{\pm1}$ and $b$ to $a_j^{\pm1}$; the sign of the third
image can always be chosen to make the determinant one. A positive
image of $a$ gives $R_{ij}^{\pm1}$, while a negative image gives
$L_{ij}^{\pm1}$, with the exponent controlled by the sign of the
image of $b$. Gersten's theorem therefore puts $\SAut(F_3)$ in
$\langle\alpha,h_\delta\rangle$. The determinants of the generators
give $\langle\alpha,h_\delta\rangle=G_{3,\delta}$.

To recover the conjugator using only $h_\delta,z$, put
$p=zh_\delta zh_\delta^{-1}$, $v=h_\delta^{-1}zh_\delta$ and $W=[p,v]$.
Their actions are independent of $\delta$:
\begin{equation}\label{eq:rank3-membership-table}
\begin{array}{c|ccc}
 &a&b&c\\ \hline
p&a&b^{-1}&ac^{-1}a^{-1}\\
v&ab^{-2}&b^{-1}&c^{-1}\\
W&ab^4&b&b^{-2}cb^2.
\end{array}
\end{equation}
The intermediate word has an inner error: if $q_0$ sends $a$ to
$b^2ab^2$ and fixes $b,c$, then $W=\kappa_{b^{-2}}q_0$, where
$\kappa_w$ is conjugation by $w$. The product below cancels that error
and gives the exact identity
\begin{equation}\label{eq:rank3-membership}
q=WzWz\in\langle h_\delta,z\rangle,\qquad [q,z]=1.
\end{equation}
For example, $zWz$ sends $(a,b,c)$ to $(b^4a,b,b^2cb^{-2})$,
and applying $W$ gives $(b^4ab^4,b,c)$. The commutation relation follows
from the defining images of $q,z$. Theorem~\ref{thm:criterion}(1),
with $x=h_\delta$, $r=\alpha$ and $m=2$, now proves the lemma.
\end{proof}

\begin{lemma}[Separation]\label{lem:rank3-trace}
For both values of $\delta$,
\begin{equation}\label{eq:rank3-polynomial}
J_\rho(\mathcal P_{3,\delta,t})=128t^2+16t.
\end{equation}
\end{lemma}
\begin{proof}
In the homology basis $[a],[b],[c]$,
\begin{equation}\label{eq:rank3-matrices}
A_\delta=\begin{pmatrix}1&1&0\\0&0&-1\\s&0&0\end{pmatrix},\qquad
B=\begin{pmatrix}0&-1&0\\1&0&0\\0&0&1\end{pmatrix},\qquad
\rho(q)=I_3+8E_{21}.
\end{equation}
Put $k=8t$. Then $B_t=\rho(\alpha_t)$ and
$C_{\delta,t}=A_\delta B_tA_\delta^{-1}B_t^{-1}$ are
\begin{equation}\label{eq:rank3-commutator}
B_t=\begin{pmatrix}k&-1&0\\k^2+1&-k&0\\0&0&1\end{pmatrix},\qquad
C_{\delta,t}=\begin{pmatrix}
k^2+k&-k-1&s(k^2+2k+2)\\
-k^2-1&k&0\\
sk&-s&k+1
\end{pmatrix}.
\end{equation}
The determinant of $C_{\delta,t}$ is $1$, and its diagonal cofactors
are $k^2+k,-k,-k-1$. Consequently
\begin{equation}\label{eq:rank3-traces}
\tr C_{\delta,t}=64t^2+24t+1,\qquad
\tr C_{\delta,t}^{-1}=64t^2-8t-1.
\end{equation}
Adding proves \eqref{eq:rank3-polynomial}, independently of $\delta$.
\end{proof}

The two rank ranges now give the four geometric groups simultaneously.
\begin{proposition}\label{prop:free}
For every $n\ge3$ and $\delta\in\{0,1\}$, the pairs
$\mathcal P_{n,\delta,t}$ are pairwise Nielsen-inequivalent generating
pairs of $G_{n,\delta}$, taking $t\ge1$ for $n\ge4$ and $t\ge0$ for
$n=3$. Their outer images have the same properties in
$\overline G_{n,\delta}$.
\end{proposition}
\begin{proof}
Generation is given by Lemmas~\ref{lem:large-generation}
and~\ref{lem:rank3-generation}. The polynomials in
Lemmas~\ref{lem:large-trace} and~\ref{lem:rank3-trace}, denoted by
$J_n(t)$, satisfy
\[
J_n(t+1)-J_n(t)=\begin{cases}
256t+144>0,&n=3,\quad t\ge0,\\
-8t<0,&n\ge4,\quad t\ge1.
\end{cases}
\]
Thus Lemma~\ref{lem:trace}(1) separates every pair within each signed
family. The projection onto $\overline G_{n,\delta}$ preserves
generation, and the homology representation factors through it, so the
same trace values separate the outer images.
\end{proof}

\begin{corollary}[Arithmetic consequences]\label{cor:arithmetic}
The groups $\GL_n(\Z)$ and $\SL_n(\Z)$ have countably infinitely many
Nielsen classes of generating pairs for every $n\ge3$.
\end{corollary}
\begin{proof}
The homology maps $\Out(F_n)\twoheadrightarrow\GL_n(\Z)$ and
$\SOut(F_n)\twoheadrightarrow\SL_n(\Z)$ are surjective. Apply
Remark~\ref{rem:quotients} to the outer pairs in Proposition~\ref{prop:free}:
the same nonconstant trace polynomials separate their generating
homology images. Both linear groups are countable.
\end{proof}

\begin{proof}[Proof of Theorem~\ref{thm:main}]
Proposition~\ref{prop:surface} and Corollary~\ref{cor:symplectic} give
part~(1). Proposition~\ref{prop:free} and Corollary~\ref{cor:arithmetic}
give part~(2). Each group is countable, so its infinitely many Nielsen
classes form a countably infinite set.
\end{proof}

\section*{AI Declaration}
Generative-AI systems, specifically GPT-6 astra and GPT-6.1 sol, were
used as research tools in the development of some of the mathematical
arguments in this paper. In an early stage of the project, an AI system produced an
initial proof candidate for the mapping-class-group result, which the
author subsequently verified and reorganized. From its
underlying mechanism, the author formulated the abstract root-conjugation criterion
of Theorem~\ref{thm:criterion}. The author then used this criterion to
obtain the $\SL_n(\Z)$ and $\GL_n(\Z)$ applications. Generative-AI
systems were also used in later exploratory work that led to the
constructions for $\Aut(F_n)$, $\SAut(F_n)$, $\Out(F_n)$, and
$\SOut(F_n)$. The final proofs, statements, references, and exposition
were checked and revised by the author, who assumes full responsibility
for the contents of the paper.

\end{document}